\documentclass{amsart}

\usepackage{amsmath}
\usepackage{amsfonts}
\usepackage{amssymb}
\usepackage{amsthm}
\usepackage{mathtools}
\usepackage{enumerate}
\usepackage{hyperref}
\usepackage[all]{xy}
\usepackage{prettyref}
\usepackage{adjustbox}
\usepackage{color}
\usepackage{lscape}

\usepackage{tikz}
\usetikzlibrary{shapes,arrows}

\newcommand{\ZZ}{\mathbb Z}
\newcommand{\QQ}{\mathbb{Q}}
\newcommand{\CC}{\mathbb{C}}
\newcommand{\kbar}{\overline{k}}
\newcommand{\PP}{\mathbb P}
\newcommand{\FF}{\mathbb F}

\newcommand{\cO}{\mathcal{O}}
\newcommand{\Xbar}{\overline{X}}

\DeclareMathOperator{\HH}{H}
\DeclareMathOperator{\Pic}{Pic}
\DeclareMathOperator{\Div}{Div}

\DeclareMathOperator{\Aut}{Aut}
\DeclareMathOperator{\Gal}{Gal}

\DeclareMathOperator{\rk}{rk}

\theoremstyle{plain}
\newtheorem{theorem}{Theorem}[section]
\newtheorem{corollary}[theorem]{Corollary}
\newtheorem{proposition}[theorem]{Proposition}
\newtheorem{lemma}[theorem]{Lemma}

\newtheorem{algorithm}[theorem]{Algorithm}

\theoremstyle{definition}

\newtheorem{example}[theorem]{Example}
\newtheorem{Input}[theorem]{Input}
\newtheorem{Output}[theorem]{Output}

\newtheorem{remark}[theorem]{Remark}

\title[]{A practical algorithm to compute the geometric Picard lattice of K3 surfaces of degree $2$}

\author[]{Dino Festi}

\date{}

\address[Dino Festi]{Institut f\"{u}r Mathematik \\ Johannes Gutenberg--Universit\"{a}t \\
Staudingerweg~9, 
55128 Mainz,
Germany }

\email{dinofesti@gmail.com}
\begin{document}

\begin{abstract}
Let $k$ be either a number a field or a function field over $\QQ$ with finitely many variables.
We present a practical algorithm to compute the geometric Picard lattice of a K3 surface over $k$ of degree $2$, i.e., 
a double cover of the projective plane over $k$ ramified above a smooth sextic curve.
The algorithm might not terminate;
if it terminates then it returns a proven correct answer.
\end{abstract}

\maketitle

\section{Introduction}

The Picard lattice of a K3 surface is a powerful tool to understand both the arithmetic and the geometry of a K3 surface.
From a geometric point of view, by looking at its Picard lattice we can read off whether a surface is Kummer or not or
whether it admits an elliptic fibration.
From an arithmetic point of view we can sometimes see whether rational points on the surface are potentially dense,
or gather information about the Brauer group of the surface and hence on the Brauer--Manin obstruction for the existence of rational points on the surface.

Given the importance of this invariant,
much effort has been put in finding ways to compute it.
Ronald van Luijk, in his PhD thesis (cf.~\cite{vL05}), presents a practical method to give an upper bound for the rank of the Picard lattice of a K3 surface over a number field.
In~\cite{HKT13}, Hassett, Kresch and Tschinkel
give a theoretical algorithm 
that in a priori bounded time returns the geometric Picard lattice of a given K3 surface of degree $2$.
A more general conditional algorithm to compute the N\'eron--Severi  group of any smooth projective integral variety is given by Bjorn Poonen, Damiano Testa, and Ronald van Luijk in~\cite{PTvL15}.
For a more detailed account of the progress in this matter,
see \cite[Section 2]{PTvL15}.

Despite all this effort, 
we do not have yet a practical algorithm that, 
given the equation defining a K3 surface as input,
returns the Picard lattice as output.
In fact, none of the algorithms presented so far has been ever used in practice, 
let alone implemented in any computer algebra.

In this paper we present a practical algorithm to compute the Picard lattice of a K3 surface of degree $2$ over a field $k$ that is either a number field or a function field over $\QQ$ with finitely many variables.
This algorithm might not terminate (cf. Section~\ref{s:Issues});
if it is forced to terminate 
(because of time limits, for example),
then it still returns a sublattice of the geometric Picard lattice of the surface:
this sublattice can still be used for some of the applications mentioned before, e.g., checking the existence of elliptic fibrations or potential density of rational points,
and in fact this sublattice might still turn out to be the full geometric Picard lattice.
If the algorithm terminates, 
it returns the full geometric Picard lattice.
With little extra work, the same algorithm can  also be applied to K3 surfaces obtained as resolution of double cover of the projective plane ramified above a sextic curve with du Val singularities  (see Remark~\ref{r:Singularities})
and, as remarked by Alice Garbagnati, 
to K3 surfaces given by quartic surfaces of $\PP^3$ admitting one node (see Remark~\ref{r:Qn}).
Algorithm~\ref{a:Main} is practical in the sense that it has been successfully used in practice, see~\cite{BCFNW18}.
The variation for double covers of $\PP^2$ ramified above singular sextics  has been successfully used in~\cite{FvS18}.
An implementation of the algorithm in \texttt{MAGMA} is currently work in progress.

The material in this paper is mostly not original:
it is rather the  rearrangement of existing results and techniques in a coherent manner.

\begin{algorithm}\label{a:Main}
The algorithm takes the field $k$ and the surface $X$ defined by the equation
$$
w^2=f(x,y,z)
$$
as input and, 
after performing the following four steps,
it returns the Gram matrix of the geometric Picard lattice of $X$.
\begin{enumerate}[I.]
\item Give an upper bound for $\rho (\Xbar )$.
\item Find explicit divisors on $\Xbar$.
\item Compute the sublattice $\Lambda$ of $\Pic \Xbar$ generated by the divisors found in Step II.
 \item Check whether $\Lambda = \Pic \Xbar$.
\end{enumerate}
\end{algorithm}

\begin{remark}
We want to stress out that the algorithm could fail in providing the Picard lattice of the given surface.
The critical point of the algorithm is usually Step II.
Step IV can in principle also present some fatal obstructions.
Step I and III can be computationally expensive,
but usually do not represent major obstacles.
More details on these issues can be found in Section~\ref{s:Issues}.
\end{remark}

\begin{remark}\label{r:Qn}
As remarked by Alice Garbagnati, 
Algorithm~\ref{a:Main} can also be applied to K3 surfaces given by a quartic in $\PP^3$ admitting a node, which form an $18$-dimensional space.
To see this, let $X\subset \PP^3$ be such a surface and,
without loss of generality,
assume that the node lies at the point $P=(0:0:0:1)$.
If $s,t,u,v$ are the coordinates of $\PP^3$,
then $X$ is defined by the equation $f=0$ with $f$ of the form
$$
f=f_4(s,t,u)+f_3(s,t,u)v+f_2(s,t,u)v^2,
$$
with $f_i$ a homogeneous polynomial of degree $i$, for $i=2,3,4$.
The projection from $P$ gives a birational map from $X$ to 
the double cover of $\PP^2$ ramified above the sextic given by
$f_3^2-4f_2f_4=0$.
More precisely, 
if $x,y,z,w$ are the coordinates of the weighted projective space $\PP(1,1,1,3)$,
and $Y$ is the double sextic given by $w^2=f_3^2(x,y,z)-4f_2(x,y,z)f_4(x,y,z)$,
then the projection from $P$ gives a birational map from $X$ to $Y$;
this map is only defined up to a sign.
Obviously, Algorithm~\ref{a:Main} can be applied to $Y$.
In this view, notice that the plane conic $C\colon f_2=0$ is everywhere tangent to the branch locus of $Y$ (cf. Example~\ref{e:sixtangent}).
\end{remark}

\begin{remark}
According to some definitions (for example, Knuth's),
an algorithm is bound to finish in finite time.
Therefore, in this respect, Algorithm~\ref{a:Main} would not deserve its title.
Nevertheless, 
notice that in practice one can put time constrains to the algorithm, 
forcing it to terminate after a given amount of time or steps.
Also in the case of a forced arrest, 
the algorithm is still  able to return a lattice as output,
together with a message indicating that the returned lattice needs not to be the full geometric Picard lattice.
Therefore Algorithm~\ref{a:Main}, in its practical applications, also satisfies those definitions of algorithm requiring finiteness 
(and in particular Knuth's definition).
\end{remark}

In Section~\ref{s:Def}  we introduce basic definitions and facts about K3 surfaces in general and their Picard lattice in particular.
The specifications of the input and output of the algorithm are given in Section~\ref{s:InOut}.
In Sections~\ref{s:Bound}--\ref{s:Check} we illustrate the four steps of the algorithm.
Finally, in Section~\ref{s:Issues}, 
we explain how the algorithm can fail in providing the full  geometric Picard lattice of a given surface.

\section*{Acknowledgements}
Thanks are due to Ronald van Luijk for the numerous discussions on this topic,
and Duco van Straten for the suggestions on the part regarding singular surfaces.
The author would also like to thank Edgar Costa, Alice Garbagnati, and Yuri Tschinkel for useful comments.
This paper has been  written during the author' stay in Mainz, 
supported by SFB/TRR 45.

\section{Basic definitions and facts}\label{s:Def}

In this section we define the terms appearing in the title
and some other objects that will be often used in the remainder of the paper.
We will then introduce some results about K3 surfaces and their Picard lattice that are already well known and that we will freely use in the remainder of the paper.

Let $k$ be any field.
In this paper, by \textit{surface} over $k$ we mean 
a separated, geometrically integral scheme of finite type and dimension two over $k$.
Let $X$ be a surface over $k$.
We say that $X$ is a \textit{K3 surface} if it is smooth, projective and such that 
$$
\HH^1(X,\cO_X)=0 \; \text{ and } \; K_X\sim_{\text{lin}} 0,
$$
where $\sim_{\text{lin}}$ denotes the linear equivalence of divisors;
the last two conditions are equivalent to asking $X$ to be
simply connected and such that
$$ \omega_X:=\Omega_X^2 \cong \cO_X.$$

Let $X$ be a K3 surface over a field $k$.
We define the \textit{Picard group} of $X$, 
denoted by $\Pic X$, 
to be the quotient of the divisor group $\Div X$ modulo linear equivalence.
Equivalently, one could define $\Pic X$ as the group of isomorphism classes of invertible sheaves on $X$; 
finally, it turns out that  $\Pic X \cong \HH^1(X,\cO^*_X)$.

Let $\kbar$ be an algebraic closure of the field $k$;
we denote by $\Xbar:=X\times_k \kbar$ the base-change of $X$ to $\kbar$. 
We define the \textit{geometric Picard group} of $X$ to be
$\Pic \Xbar = \Div \Xbar /\sim_\text{lin}$.

We define a \textit{lattice} to be a finitely generated, torsion-free
$\ZZ$-module endowed with a non-degenerate, symmetric bilinear form.

The intersection pairing on $\Div X$ induces a pairing on $\Pic X$,
also called intersection pairing.
We will see in the next section that $\Pic X$ endowed with the intersection pairing turns out to be a lattice, called the \textit{Picard lattice} of $X$.

We define the \textit{Picard number} of $X$, 
denoted by $\rho(X)$, 
to be the rank of the Picard lattice $\Pic X$ of $X$;
we define the \textit{geometric Picard number} of $X$,
denoted by $\rho (\Xbar )$, 
to be the rank of the \textit{geometric Picard lattice} $\Pic \Xbar$ of $X$.

A \textit{polarized} K3 surface is a K3 surface $X$ together with an ample divisor $L$;
the \textit{degree} of a polarized K3 surface is the self intersection $L^2$ of the ample divisor $L$.
The \textit{genus} of a polarized K3 surface is defined to be the genus of a curve inside the linear system $|L|$.

\begin{example}
Smooth quartics in $\PP^3$ are  K3 surfaces
 of degree $4$ and genus $3$.
\end{example}
\begin{example}
Double covers  $\pi\colon X\to \PP^2$ ramified above a smooth sextic curve $B$ are K3 surfaces of degree $2$ and genus $2$.
They are also called double sextics.
If $B$ is defined by the homogeneous sextic polynomial $f$,
then $X$ can be defined by the equation $w^2=f(x,y,z)$ inside the weighted projective space $\PP (1,1,1,3)$ with coordinates $x,y,z,w$.
These are the surface we will focus on in this paper.
\end{example}

Finally, we define the \textit{$K3$-lattice} $\Lambda_{K3}$ to be the lattice 
$U^{\oplus 3} \oplus E_8(-1)^{\oplus 2}$,
where $U$ is the hyperbolic lattice given by $\ZZ^2$ with intersection matrix 
$\bigl(\begin{smallmatrix}
0&1 \\ 1&0
\end{smallmatrix} \bigr)$,
and $E_8(-1)$ is the lattice associated to $\ZZ^8$ 
and having as intersection matrix the matrix associated to the Dynkin diagram $E_8$ multiplied by $-1$.

The following facts are well known, and we refer to~\cite{Huy16} for the proofs.
\begin{proposition}\label{p:basicsPic}
Let $X$ be a K3 surface defined over $\CC$.
Keeping the notation introduced before, the following statements hold.
\begin{enumerate}
\item $\Pic X$ is a finitely generated, torsion free abelian group.
\item The intersection pairing on $\Pic X$ is even, non-degenerate, symmetric and bilinear.
\item The cohomology group $\HH^2(X,\ZZ)$ endowed with the cup product is a lattice and it is isometric to
$\Lambda_{K3}$.
\item There is an injection of lattices $\Pic X \hookrightarrow \HH^2(X,\ZZ)$.
\item In particular, $\Pic X$  endowed with the intersection pairing is a lattice of rank $\rho (X)\leq 22$.
\item $\Pic X$ is an even lattice of signature $(1,\rho (X)-1)$.
\item The injection $\Pic X \hookrightarrow H^2(X,\ZZ)$ in (4)
induces an isometry $\Pic X \cong H^{1,1}(X) \cap H^2(X,\ZZ )$.
In particular, $\rho(X)\leq 20$.
\end{enumerate}
\end{proposition}
\begin{proof}
\begin{enumerate}
\item It follows from~\cite[Proposition 1.2.1]{Huy16} and~\cite[Proposition 1.2.4]{Huy16}.
\item Bilinearity follows immediately from the definition of the pairing (cf.~\cite[1.2.1]{Huy16});
for the other properties see~\cite[Proposition 1.2.4]{Huy16}.
\item \cite[Proposition.1.3.5]{Huy16}.
\item \cite[1.3.2]{Huy16}.
\item It follows directly from the previous points.
\item \cite[Proposition 1.2.4]{Huy16}.
\item \cite[1.3.3]{Huy16}.
\end{enumerate}
\end{proof}

\section{Input \& Output}\label{s:InOut}
The goal of the algorithm is to compute $\Pic \Xbar$, that is, 
compute its Gram matrix.
In order to run the algorithm, we need some assumptions on the input.
We assume that $k$ is either a number a field or a function field over $\QQ$ with finitely many variables.
We also assume that $k$ is explicitly given.
The examples to keep in mind are number fields with a given generator, and function fields with a finite number of given variables.

\begin{Input}
$(k, f, \text{ (optional) } G, \text{ (optional) } T)$.

We assume that $X$  is a K3 surface over $k$ given as the hypersurface
$$
w^2=f(x,y,z)
$$
inside the weighted projective space $\PP=\PP(1,1,1,3)$ over $k$, 
with coordinates $x,y,z,w$ of weights $1,1,1,3$, respectively.
We assume that $f$ is a homogeneous polynomial of degree $6$ with coefficients in $k$ defining a smooth sextic curve in $\PP^2_k$.

Optionally, one can provide a subgroup $G\leq \Aut (\Xbar )$ of automorphisms of $\Xbar$ as part of the input.
These automorphisms have to be given in the coordinates of $\PP$, that is, as projective transformations of $\PP$.
If no specification for $G$ is given,
then $G$ is considered to be the trivial group.  

In order to make sure that the algorithm terminates,
one can provide a time bound $T$ for the running time of the algorithm.
If no specification for $T$ is given, 
then $T$ is considered to be $\infty$.
\end{Input}

\begin{Output}
$((M), \texttt{true/false}) $. 

The output of the algorithm consists of a square matrix $M$
representing  the intersection matrix of the lattice generated by the divisors found by the algorithm, together with a boolean expression: \texttt{true} or \texttt{false}.

If the boolean statement is \texttt{true},
then $M$ is the Gram matrix of the full geometric Picard lattice of the surface given as input.
Otherwise, the algorithm only produced a sublattice of $\Pic \Xbar$,
with no guarantee it is the full geometric Picard lattice. 
\end{Output}

\begin{remark}
We have seen that if the algorithm terminates, then
the output  is a square matrix representing the Gram matrix of the sublattice $\Lambda\subseteq \Pic \Xbar$ computed by the algorithm.
For an easier reading of an output,
sometimes it might be desirable to write $\Lambda$ as direct sum of known lattices,
that is,
as a tuple of the form $((M_1,m_1),...,(M_t,m_t))$, with $M_i$ either equal to $U$ or a lattice of $ADE$--type  such that 
$M$ is isometric to the lattice
$$
M_1(m_1) \oplus ... \oplus M_t (m_t).
$$
This goal can be achieved when, for instance,
$\rk\Lambda \geq \ell (A_\Lambda +2)$,
where $\ell (A_\Lambda +2)$ denotes the minimal number of generators of the discriminant group $A_\Lambda$ of $\Lambda$.
In this case, one can then look for tuples $((M_1,m_1),...,(M_t,m_t))$ such that the lattice $M_1(m_1) \oplus ... \oplus M_t (m_t)$ has the same rank, signature, parity, and discriminant group as $\Lambda$.
Then the two lattices are isomorphic, as showed by \cite[Corollary 1.13.3]{Nik80}.
\end{remark}

\begin{remark}\label{r:Sing}
If $f$ defines a singular sextic curve $B$ in the plane,
then $X$ is not a K3 surface as it fails to be smooth.
Nevertheless, if $B$ has only du Val singularities, that is,
$ADE$-singularities,
then the desingularisation $\tilde{X}$ of $X$ is a K3 surface,
and this method, with some small changes, can be used to compute its geometric Picard lattice.
In fact, the exceptional divisors coming from the desingularisation contribute to the number of divisors we can find,
helping us in the generation of a big sublattice of $\Pic \overline{\tilde{X}}$.
Unfortunately, in this case computing the intersection numbers of the divisors might reveal itself more cumbersome.
See more details on this case in Section~\ref{s:Intersection};
see~\cite{FvS18} for a concrete example.
\end{remark}

\section{Step I: an upper bound for $\rho (\Xbar)$}\label{s:Bound}
Recall that the input is given by a K3 surface $X$ defined over 
a field $k$ with characteristic $0$.
Then, as shown by  Proposition~\ref{p:basicsPic}, 
a bound for the (geometric) Picard number of $X$ is given by $20$.
Of course in general this bound is far from being optimal.
In order to give a better bound, 
we know of two methods.

\subsection{van Luijk's method}\label{ss:vL} If $X$ is defined over a number field, then one can use van Luijk's method. 
This method is explained (and used) in~\cite{vL05} and~\cite{vL07}.
A refinement of van Luijk's method is presented in~\cite{EJ11}.
In general, the bound provided by this method does not need to be sharp.
In~\cite{Cha14}, Charles investigates the sharpness of this bound (cf.~\ref{ss:NonSharp}).
Notice that thanks to \cite[Proposition 3.6]{MP12}, van Luijk's method can be applied also to K3 surfaces defined over function fields.
\subsection{Dolgachev's bound} If $X$ is defined over a function field, then one can look at $X$ as the generic member of a family $\mathcal{X}$ of K3 surface;
let $d$ denote the dimension of $\mathcal{X}$ as family of $K3$ surfaces.
If the family is not isotrivial, then from~\cite[Corollary 3.2]{Dol96} it follows that the geometric Picard number of $X$ can be at most $20-d$.

\vspace*{8pt}
In what follows, we will denote by $\tau$ the bound obtained by any of the previous methods.

\section{Step II: finding divisors on $\Xbar$}\label{s:Divisors}
In this section we are going to present some tools to find divisors on a surface given by a double cover of $\PP^2$ ramified above a curve.
The case that concerns us the most is obviously when this curve is a (smooth) sextic.

\subsection{Double cover structure}\label{ss:doublesextic}

The main tool is given by the following proposition.
This result seems to be well known among the experts,
but we failed to find any early reference.
A proof of a slightly different statement by Ronald van Luijk and the author can be found in \cite[Section 5]{FvL15} and \cite[Subsection 1.2.3]{Fes16}.
For sake of clarity we will restate here the result in a more compact way.
The proof follows the proofs given in the above references.

In what follows, except when explicitly stated,
we will always assume that $X$ is a K3 surface 
over a field $k$ of characteristic $0$
given as a double cover of $\PP^2$ ramified above a smooth sextic curve $B$.
More precisely, we will assume that $X$ is given as in Section~\ref{s:InOut}, i.e.,
as
$$
X\colon w^2 = f_6(x,y,z).
$$

\begin{proposition}\label{p:Split}
Let $\pi\colon S\to \PP^2$ be a double cover of $\PP^2$ ramified above the curve $B\subset \PP^2$.
Let $C$ be a plane curve with no components in common with $B$ and with genus $g(C)=0$.
Assume the following two conditions hold:
\begin{enumerate}[(i)]
\item $C$ intersects $B$ with even multiplicity everywhere;
\item $C$ does not intersect $B$  in its singular points, i.e., 
$C^{\text{sing}}\cap B = \emptyset$.
\end{enumerate} 
Then the pre-image $\pi^{-1}(C)$ of $C$ on $S$ splits into two components.
\end{proposition}
\begin{proof}
Define $b(C):=\{T\in C(\kbar )\; :\; \mu_T (C,B) \equiv 1 \bmod 2 \}$,
where $\mu_T(C,B)$ is the intersection multiplicity of $C$ and $B$ at $T$.
Let $\theta\colon \tilde{C}\to C$ be the normalisation of $C$ and define $b(\tilde{C})$ analogously to $b(C)$.

Let $D=\pi^{-1}(C)$ be the preimage of $C$ on $X$.
Then $\pi$ induces a 2:1-morphism $\tilde{\pi}\colon \tilde{D}\to \tilde{C}$ whose branch locus is $b(\tilde{C})$ (cf.~\cite[Lemma 1.2.25]{Fes16}).

If $b(\tilde{C})=\emptyset$ then, as $g(\tilde{C})=g(C)=0$, the preimage $\tilde{D}$ is isomorphic to an unramified double cover of $\PP^1$:
the only such cover is given by a disjoint union of two copies of $\PP^1$.
This  means that $\tilde{D}$ is given by two disjoint copies of $\tilde{C}$ and hence $D$ is given by the union (not disjoint) of two components isomorphic of $C$.

We related the splitting of $D$ to the set $b(\tilde{C})$.
Now we want to relate it to the set $b(C)$.
By assumption (ii), we have that $C^{\text{sing}}\cap B = \emptyset$  and then  $b(\tilde{C})=b(C)$,
as $\tilde{C}$ and $C$ are isomorphic outside $C^{\text{sing}}$,
and all the intersection points of $C$ and $B$ lie outside that set,
and therefore the intersection multiplicity is preserved by $\theta$.
Then assumption (i) implies that $\emptyset= b(C) = b(\tilde{C})$ and the statement follows.

\end{proof}

In what follows we use $[H]$ to denote the hyperplane class of $X$ inside $\Pic X$.

\begin{corollary}
Let $C\subset \PP^2$ be a curve satisfying the hypothesis of Proposition~\ref{p:Split}.
Then the pull back $D$ on $X$ of $C$ splits into two components that are not linearly equivalent to any multiple of $[H]$,
that is, they do not lie in the span $\langle[H]\rangle_\ZZ \subset \Pic \Xbar$. 
\end{corollary}
\begin{proof}
Let $D_1, D_2$ be the two components into which $D$ splits,
i.e., $D=D_1+D_2$,
and let $d$ be the degree of $C$.
Then $H$ and $D_i$, for $i=1,2$, generate a lattice with the following Gram matrix.
$$A:=
\begin{pmatrix}
2 & d \\
d & -2
\end{pmatrix}
$$
Notice that $\det A= -4-d^2\neq 0$,
proving that $D_i$ and $H$ are linearly independent.
\end{proof}

\begin{example}[Tritangent lines]\label{e:tritangent}
It is always possible to quickly check whether there exist \textit{lines} that are everywhere tangent to $C$.
Let $\ell\subset \PP^2$  be a line,
then it is defined by an equation of the form $ax+by+cz$.
Assume $a\neq 0$, then we can write
$\ell\colon x=b'y+c'z$.
Notice that $\ell$ is obviously smooth and has genus zero.
The condition that $\ell$ intersects $C\colon f_6 (x,y,z)=0$ with even multiplicity everywhere is equivalent to 
the condition that $f(b'y+c'z,y,z)$ is the square of a homogenous cubic polynomial in $\kbar [y,z]$.
To make computations easier one can also assume $z=1$
and ask that  $f(b'y+c',y,1)$ is a square in $\kbar [y]$.

In this way we find all the tritangent lines  of the form $\ell\colon x=b'y+c'z$.
Of course, lines with $a=0$ might be tritangent too.
In order to check this, 
first assume $b\neq 0$, then write
$y=c''z$ and 
check whether $f_6(x,c''z,z)$ is a square in $\kbar [x,z]$.

Finally, we consider the line $z=0$ (obtained by assuming $a=b=0$),
and check whether $f_6(x,y,0)$ is a  square in $\kbar [x,y]$.

For more about this case, see also~\cite[Section 2]{EJ08b}.
\end{example}

\begin{example}[sixtangent conics]\label{e:sixtangent}
The procedure is analogous to the case of tritangent lines (cf. Example~\ref{e:tritangent}).
We present a refinement of the algorithm exposed in~\cite[Section 2]{EJ08b}.
Let $C$ be a smooth conic of $\PP^2$ with equation
$$
a_0+a_1x+a_2y+a_3xy+a_4x^2+a_5y^2=0,
$$
and assume that $P=(x_0,y_0)$ is a point of $C$.
One can then use $P$ to parametrise $C$, 
obtaining the map $\PP^1\to C$ defined as follows:
$$
t \mapsto (X_t: Y_t: Z_t),
$$
with 
\begin{align*}
X_t &:= a_5x_0t^2-(a_2+2a_5y_0)t-a_1-a_3y_0-a_4x_0,\\
Y_t &:= t(X_t-x_0Z_t)+y_0Z_t,\\
Z_t &:= a_5t^2+a_3t+a_4.
\end{align*}
Recall that the branch locus of $X\to \PP^2$ is the sextic $B\colon f_6(x,y,z)=0$.
Then $C$ is tangent everywhere to $B$ if and only if 
$f_6(X_t,Y_t,Z_t)$ is a square in $\kbar [t]$.
\end{example}
\begin{remark}
The condition that $f_6(X_t,Y_t,Z_t)$ is a square in $\kbar [t]$ consists  of asking that $f_6(X_t,Y_t,Z_t)=(b_0t^6+...+b_6)$,
for some $b_0,...,b_6\in \kbar$.
This means that finding sixtangent conics implies finding a Gr\"obner basis of an ideal of a polynomial ring of 14 variables over $k$ generated by seven elements.
This is the first computationally very expensive step.
\end{remark}

\begin{example}[Curves of degree $d$ with a singular point of multiplicity $d-1$]
One can generalise Example~\ref{e:sixtangent} to any curve $C$ of degree $d$ with a singular point $P$ of multiplicity $d-1$.
In fact $C$ has genus $0$ and the projection from $P$ provides a parametrisation.
Asking $C$ to be everywhere tangent to $B$, i.e., $3d$--tangent to $B$, again boils down to checking whether the polynomial
$f_6(X_t,Y_t,Z_t)$ is a perfect square inside $\kbar [t]$.

Although on paper this is conceptually easy,
from a computational point of view this approach becomes quickly infeasible:
in general, the computations need to be done on a polynomial ring with $5d+4$ variables.
\end{example}

\begin{example}[Curves of degree $d\geq 4$ and genus $0$]
Curves of degree $d$ with a singular point of multiplicity $d-1$ are only one example of curves of degree $d$ and genus $0$.
In general, a curve of degree $d$ and genus $0$ can have $r$ singular points $P_1,...,P_r$ of multiplicity $m_1,...,m_r$ respectively such that
$$
\frac{(d-1)(d-2)}{2}=\sum_{i=1}^r\frac{m_i(m_i-1)}{2}.
$$
This means that in looking for divisors,
one can ultimately look for curves of arbitrary high degree with compatibly many singular points, provided these singular points do not lie on $B$.
On the one hand, this gives an enormous space of curves to look into and hence many possibilities to find such curves;
on the other hand, this space is so big that any extensive brute-force search is infeasible.
\end{example}

\subsection{Del Pezzo surface of degree $1$}\label{ss:doubledP}

Recall that $X\to \PP^2$ is the double cover of the plane ramified above the smooth sextic curve $B$ defined by  $f(x,y,z)=0$.
Sometimes, a double cover of the projective plane ramified along a smooth sextic can have the additional structure  of double cover of a del Pezzo surface of degree~$1$.
This structure is encoded by a very easy condition on the equation defining the smooth plane sextic,
as shown by the following result.

\begin{proposition}
Assume that in $f$ there is a variable, say $x$, appearing with even powers;
in other words, $f(x,y,z)=f'(x^2,y,z)$ for some polynomial $f'$.
Then $X$ is the double cover of the surface $X':v^2=f(u,s,t)$ inside the weighted projective space $\PP (3,1,1,2)$ with coordinates $u,s,t,v$.
The surface $X'$ is a del Pezzo surface of degree 1.
\end{proposition}
\begin{proof}
To prove that $X$ is a double cover of $X'$ 
just notice that the map 
$$(x:y:z:w)\mapsto (x^2:y:z:w)$$ 
is well defined and 2-to-1.

The second part of the statement follows directly from~\cite[Theorem III.3.5]{Kol96}.
\end{proof}

\begin{corollary}
Assume $f(x,y,z)=f'(x^2,y,z)$ for some polynomial $f'$.
Then $\rho (\Xbar )\geq 9$.
\end{corollary}
\begin{proof}
By hypothesis $X$ is a double cover of a del Pezzo surface of degree $1$, say $\pi \colon X\to Y$.
It is well known that the Picard group of a del Pezzo surface of degree $1$ is isomorphic to $\ZZ^{\oplus 9}$ (see, for example,~\cite[Proposition IV.25.1]{Man86}).
Let $L_1,...,L_9$ be the nine generators of the geometric Picard lattice of $Y$.
Then the pull-backs $\pi^*L_1,...,\pi^*L_9$ are nine linearly independent classes in $\Pic \Xbar$,
proving the statement.
\end{proof}

\subsection{Propagating the divisors}

Let $\Sigma_0$ be a set of divisors of $\Xbar$ (not necessarily obtained by the methods exposed in 
Subsections~\ref{ss:doublesextic} and~\ref{ss:doubledP}). 
Assume all the divisors in $\Sigma_0$ are  given by explicit equations.

If a subgroup $G\subseteq \Aut \Xbar$ is given in the input,
then we can let $G$ act on $\Sigma_0$,
obtaining the set $G\cdot \Sigma_0 =: \Sigma_1$.

Furthermore, from the procedure we see that the divisors in $\Sigma_0$ need not to be defined over $k$,
and the same holds a priori for the automorphisms in $G$.
So let $K\supset k$ be the minimal Galois extension such that all the elements of $\Sigma_1$ can be defined over $K$.
Then the Galois group $\Gal (K/k)$ acts on $\Sigma_1$ by acting on the coefficients of the equations defining the divisors.
Let $\Sigma$ denote the set $\Gal (K/k)\cdot \Sigma_1$ obtained by letting $\Gal (K/k)$ act on $\Sigma_1$,
and let $\Lambda:=\langle \Sigma \rangle_\ZZ$ be the sublattice of $\Pic \Xbar$ spanned (over $\ZZ$) by the elements of $\Sigma$.  

\begin{lemma}
The sublattice $\Lambda\subset \Pic \Xbar$ is stable under the action of $G$ and $\Gal (K/k)$.
\end{lemma}
\begin{proof}
Both $G$ and $\Gal (K/k)$ act as permutation groups of the generators of $\Lambda$.
\end{proof}
\begin{remark}
Notice that $\Sigma$ is a set of generators for $\Lambda$,
but it does not need to be minimal.
In fact, 
we have seen that as soon as 
we include the hyperplane section and 
the two components of the pull-back of plane curve everywhere tangent to $B$,
then these three divisors are linearly equivalent.
\end{remark}

\section{Step III: computing the intersection numbers}\label{s:Intersection}

Assume that using the methods presented in 
Section~\ref{s:Divisors}
we get a set of divisors $\Sigma=\{ H, D_1,...,D_n \}$,
where $H$ is the hyperplane section.
We need now to compute the intersection matrix of these divisors.

As all the divisors come from explicit equations, 
this can easily be done by simply computing 
the number of points, with multiplicity,
of the (zero-dimensional) intersection $D_i \cap D_j$.
This task is easily taken care of by a computer.

Computing the self-intersection is even easier,
as shown by the following lemma.
\begin{lemma}
For each $i=1,...,n$ we have $D_i^2=-2$.
Furthermore, $H^2=2$.
\end{lemma}
\begin{proof}
For every $i=1,...,n$, 
the divisor $D_i$ is by construction 
isomorphic to a genus $0$ curve.
Then, using the adjunction formula (cf.~\cite[Proposition V.1.5]{Har77}), 
keeping in mind that $X$ is a K3 surface and hence $K_X=0$, 
it follows that $D_i^2=-2$.
\end{proof}

Let $A$ denote the $n\times n$ intersection matrix obtained in this way.
Then one can define $\Lambda$ to be the lattice having Gram matrix $A$.
Then one can ask for a minimal set of generators,
say $\{ L_1,...,L_r\}$. 
Using this set of generators one can immediately compute the rank of $\Lambda$, namely $r$,
the Gram matrix $M:=(L_i\cdot L_j)_{1\leq i,j\leq r}$,
and the determinant of $\Lambda$, by definition $\det M$.

\begin{remark}\label{r:Singularities}
Things become more cumbersome if we allow $B$ to have du Val singularities;
in particular, let us consider the case of $B$ having singularities of $A$--type.
In this case the surface $X$ given by the double cover of $\PP^2$ ramified above $C$ turns out to be a singular surface, with singularities of $A$--type above the singular points of $B$.
By resolving the singularities we obtain the K3 surface $\tilde{X}$.
If $P\in X$ is a singular point of type $A_n$,
then on $\tilde{X}$ there are $n$ exceptional divisor,
whose intersection matrix is exactly the matrix $A_n$.
These exceptional divisors are pairwise conjugated under the double-cover involution of $\tilde{X}$.
This ambiguity affects the computation of the intersection numbers of the exceptional divisors with the divisors $D_1,...,D_n$,
considering that in general computing an explicit projective model for $\tilde{X}$ is not computationally friendly. 
A way to overcome this problem is to label the exceptional divisors according to the intersection with the divisors $D_i$'s.
Despite this trick, there might still be some ambiguity left.
In this case, one can just try all the possible combinations of intersection numbers (they are bounded) 
and rule out the combinations that return an intersection matrix with rank exceeding $\tau$.
An explicit application of this method is given in~\cite{FvS18}.
\end{remark}

\section{Step IV: Check whether $\Lambda = \Pic \Xbar$}\label{s:Check}

In order to perform this step we have to assume that $r$ is in fact
equal to $\rho (\Xbar)$.
This assumption is fulfilled if, for example, we have $r=\tau$,
where $\tau$ is the upper bound for $\rho (\Xbar)$ obtained in Section~\ref{s:Bound}.
The strategy used in this step is inspired by the proof of~\cite[Theorem 7]{ST10}.
Throughout this section we will freely use the results about lattice theory exposed in~\cite[Section 1.1]{Fes16} and 
we will use simply $P$ to denote $\Pic \Xbar$.
Let $T=T(\Xbar)$ denote the geometric transcendental lattice of $X$, that is,
the orthogonal complement $P^{\perp}\subseteq H^2(\Xbar,\ZZ)$ of $P$ inside $H^2(X,\ZZ)\cong \Lambda_{K3}$.
If $L$ is a lattice, we define the discriminant group $A_L$ as the quotient $L^*/L$, where $L^*$ is the dual lattice of $L$.

Let $\iota \colon \Lambda \hookrightarrow P$ the inclusion map;
as $r=\tau$, the map $\iota$ has finite cokernel, i.e., 
the index $[P: \Lambda]$ is finite, 
and $\det \Lambda = [P:\Lambda]^2\det P$. 

We have to preliminary tests to check whether $\Lambda=P$,
as shown by the following proposition.
\begin{proposition}\label{p:FirstChecks}
Let $P, T$, and $\Lambda$ defined as above.
The following statements hold:
\begin{enumerate}[i)]
\item If $\det \Lambda$ is square-free, then $P=\Lambda$;
\item if the length $\ell(A_\Lambda)$ of the discriminant group $A_\Lambda$ is greater than $22-r$,
then $\Lambda\neq P$.
\end{enumerate}
\end{proposition}
\begin{proof}
\begin{enumerate}[i)]
\item The statement follows directly from  $\det \Lambda = [P:\Lambda]^2\det P$.
\item As $T$ is the orthogonal complement of $P$ inside $H^2(\Xbar,\ZZ)$, 
it follows that $\rk T = \rk H^2 - \rk P$.
It is known that $H^2(\Xbar,\ZZ)$ is isomoprhic to the K3 lattice $\Lambda_{K3}$, which  has rank $22$;
by initial assumption we have $\rk P = \rk \Lambda = r$.
It follows that $\rk T = 22-r$.
As $\Lambda_{K3}$ is unimodular,
$A_P\cong A_T$.
It is a well known result that the length of the discriminant group of a lattice cannot be greater than the rank of the lattice,
and so
$$
\ell(A_{P})=\ell (A_T)\leq \rk T = 22-r,
$$ 
proving the statement.
\end{enumerate}
\end{proof}

\begin{remark}
Needless to repeat, 
in case during the algorithm Propsition~\ref{p:FirstChecks}(ii) tells us that our lattice $\Lambda$ cannot be the full geometric Picard lattice,
we ought to go back to Step II and look for more divisors. 
\end{remark}
\begin{remark}
Of course, as $\ell (A_\Lambda) \leq r$,
the second statement of the proposition is only useful when $r>11$.
\end{remark}
 
 In most cases, Proposition~\ref{p:FirstChecks} will not return a positive answer to the check.
 This is why we need a finer strategy to prove $\Lambda=P$.
Let $p$ be any prime.
Then the map $\iota$ induces the map
$$
\iota_p\colon \Lambda/p\Lambda \to P/pP.
$$
Notice that $\iota_p$ is injective if $p^2$ does not divide $\det \Lambda$,
otherwise it does not need to be so.

\begin{proposition}
The map $\iota_p$ is injective for every prime $p$ such that $p^2|\det \Lambda$ if and only if $P=\Lambda$.
\end{proposition}
\begin{proof}
Notice that if $P=\Lambda$, then $\iota_p$ is trivially injective as it reduces to the identity.
To see the converse, notice that
the kernel of $\iota_p$ is the quotient $\frac{\Lambda\cap pP}{p\Lambda}$.
By hypothesis this group is always trivial.
This means that there are no elements of $\Lambda$ that are $p$-divisible in $P$ but not in $\Lambda$,
i.e., $\Lambda$ is primitive in $P$;
as $\Lambda$ has also finite index inside $P$,
it follows that $\Lambda=P$.
\end{proof}

As a priori we do not know what $P$ is, it is impossible for us to compute the kernel of $\iota_p$.
Nevertheless, we can define a subset of $\Lambda/p\Lambda$ that surely contains the kernel of $\iota_p$.
To do this, we need to notice that the intersection form on $P$ induces a bilinear form $b_p\colon (\Lambda/p\Lambda)^2\to \ZZ/p\ZZ$.
We define the set $\Lambda_p$ as
$$
\{ [x]\in \Lambda/p\Lambda \;|\; \forall [y]\in \Lambda/p\Lambda 
\;\; b_p([x],[y])=0 \; \textrm{ and } \; x^2 \equiv 0 \bmod 2p^2 \}.
$$
\begin{proposition}
Let $\Lambda_p$ be defined as above. The following statements hold:
\begin{enumerate}[i)]
\item the kernel of $\iota_p$ is contained in $\Lambda_p$;
\item the set $\Lambda_p$ is fixed by all the isometries of $P$.
\end{enumerate}
\end{proposition}
\begin{proof}
Lemma~\cite[Lemma 1.1.23]{Fes16}. 
\end{proof}
\begin{remark}
As remarked in Remark~\cite[Remark 1.1.22]{Fes16},
it is very easy to compute $\Lambda_p$.
First notice that $\Lambda/p\Lambda$ is a $\FF_p$ vector space of dimension $r$, hence it is finite;
if $M$ is the Gram matrix of $\Lambda$, 
then the bilinear form on $\Lambda/p\Lambda$ is given by 
the reduction $\overline{M}$ of $M$ modulo $p$.
So $\Lambda_p$ corresponds to the elements $[x]$ in the kernel of $\overline{M}$ such that $x^2\equiv 0 \bmod 2p^2$.
The kernel of $\overline{M}$ is finite and can be easily computed;
then one can go through all the elements and check which of them satisfy also the second condition.
\end{remark}

As $\Lambda_p$ contains the kernel of $\iota_p$ but it might very well be much larger,
it might happen that $\Lambda_p$ is non trivial even though the kernel of $\iota_p$ is.
We do not have a way to completely exclude this case,
but the isometries of $P$ can be used to rule out some of these pathological cases, as shown in the following result.

\begin{proposition}\label{p:orbit}
Let $H$ be a group of isometries induced by automorphisms of $X$ and/or elements of $\Gal (K/k)$.
Let $p$ be a prime, and let $[x]$ be an element of $\Lambda_p$.
Let $V$ denote the subspace of $\Lambda/p\Lambda$ spanned the orbit of $[x]$ under the action of $H$.
Let $e$ denote the dimension of $V$.
If $p^{2e}$ does not divide $\det \Lambda$,
then $[x]$ is not in the kernel of $\iota_p$.
\end{proposition}
\begin{proof}
It follows immediately from~\cite[Proposition 1.1.28]{Fes16}.
\end{proof}

\tikzstyle{decision} = [diamond, draw, fill=green!20, 
    text width=6em, text badly centered, node distance=3cm, inner sep=0pt]
\tikzstyle{block} = [rectangle, draw, fill=blue!20, 
    text width=10em, text centered, rounded corners, minimum height=4em]
\tikzstyle{line} = [draw, -latex']
\tikzstyle{cloud} = [draw, ellipse,fill=red!20, node distance=2.5cm,
    minimum height=2em]
    
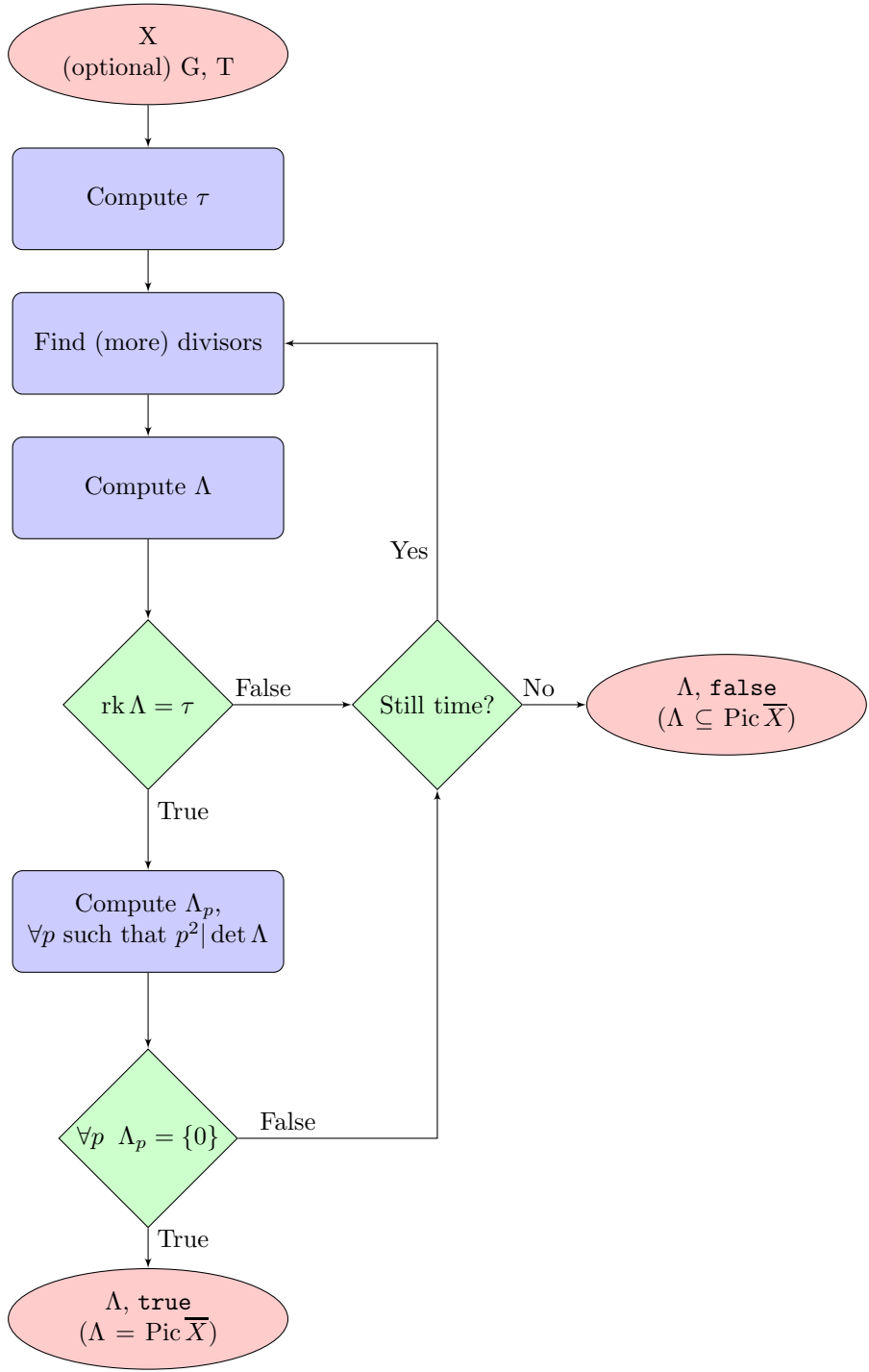
\begin{figure}
\centering
\begin{tikzpicture}[node distance = 2cm, auto]
    \node [cloud, text width = 2.5cm, align = center] (init) {X\\  (optional) G, T};
    \node [block, below of=init] (Step1) {Compute $\tau$};
    \node [block, below of=Step1] (Step2) {Find (more) divisors};
    \node [block, below of=Step2] (Step3) {Compute $\Lambda$};
    \node [decision, below of=Step3] (rank) {$\rk \Lambda = \tau$};
    \node [decision, right of=rank, node distance = 4cm] (time) {Still time?};
    \node [cloud, right of = time, text width=2.5cm,align=center, node distance = 4cm] (endtime) {$\Lambda$, \texttt{false}\\ $(\Lambda \subseteq \Pic \Xbar )$};
    \node [block, below of=rank, node distance = 3cm] (Step4) {Compute $\Lambda_p$, \\$\forall p$ such that $p^2|\det \Lambda$ };
    \node [decision, below of=Step4] (kernel) {$\forall p \;\;\Lambda_p=\{0\} $};
    \node [cloud, below of = kernel, text width=2.5cm,align=center] (end) {$\Lambda$,  \texttt{true} \\
    $(\Lambda = \Pic \Xbar)$};
    
    \path [line] (init) -- (Step1);
    \path [line] (Step1) -- (Step2);
    \path [line] (Step2) -- (Step3);
    \path [line] (Step3) -- (rank);
    \path [line] (rank) -- node [near start] {False} (time);
    \path [line] (rank) -- node [near start] {True} (Step4);
    \path [line] (time) -- (4,-4) node [near start] {Yes} --  (Step2);
    \path [line] (time) -- node [near start] {No} (endtime);
    \path [line] (Step4) -- (kernel);
    \path [line] (kernel) -- node [near start] {True} (end);
    \path [line] (kernel) --++ (4,0) node [near start] {False} -- (time);
\end{tikzpicture}
\caption{A flow chart describing Algorithm~\ref{a:Main}.}
\label{f:Main}
\end{figure}

\section{What can go wrong}\label{s:Issues}
As said at the beginning of this paper, 
Algorithm~\ref{a:Main} might not terminate.
For example, already by looking at the flow chart in Figure~\ref{f:Main}, 
one can already notice that endless loops are possible.
In this section we are going to see more in detail the possible reasons for the algorithm failing to terminate.
In some of these cases, truncating the algorithm after some time, 
the output might still turn out to be the full geometric Picard lattice,
but we would need to prove it in a different way.
This event can happen in the cases described in subsections~\ref{ss:NonSharp} and~\ref{ss:kpXL};
it can never happen in the case described in Subsection~\ref{ss:m2NoGood}.

\subsection{Non-sharp upper bound}\label{ss:NonSharp}
The first reason why the algorithm can enter into an infinite loop is given by having an upper bound $\tau$ that is larger than $\rho (\Xbar)$.
In fact, in this case, following the algorithm,
we would endlessly look for divisors in order to generate a bigger lattice and meet the upper bound.

If $\tau$ is obtained by means of van Luijk's method or any of its refinements
then, as mentioned before (cf.~\ref{ss:vL}), 
F. Charles gives a result (cf.~\cite[Theorem 1]{Cha14}) to check whether $\tau$ can be trusted as sharp or rather not.
Unfortunately, this result  relies on knowledge of the geometric transcendental lattice $T(\Xbar)$ of $X$,
that is, the orthogonal complement of $\Pic \Xbar$ inside $H^2(\Xbar,\CC)$.

Notice that if the algorithm runs in this kind of problem,
we might still have that $\Lambda=\Pic \Xbar$,
but the algorithm would fail to prove it.
Therefore, after stopping the algorithm (for example introducing a time constrain), 
we might still want to try to prove the the obtained lattice is the full geometric Picard lattice in some other ad hoc way.

\subsection{$-2$-curves might not be sufficient}\label{ss:m2NoGood}
The algorithm produces a sublattice of $\Pic \Xbar$ generated, besides the hyperplane section class, only by classes of curves of genus $0$ (that is, $-2$-curves).
There is no result stating that the Picard lattice of a K3 surface can always be generated by the classes of such curves.
Therefore, if the geometric Picard lattice of our surface $X$ cannot be generated solely by $-2$ curves and the hyperplane section class,
our algorithm is bound to never end and
any output resulting from a forced termination of the algorithm cannot be the whole geometric Picard lattice of $X$.

\subsection{$\Lambda_p$ is too large}\label{ss:kpXL}
Recall the notation introduced in Section~\ref{s:Check}.
By construction, the set $\Lambda_p$ contains the kernel of the map $\iota_p$,
and there is no reason for this inclusion to be an equality,
in fact there are examples in which this inclusion is well far from being an equality.
Therefore it might happen that, although $\Lambda$ is the full geometric Picard lattice, 
$\ker \iota_p$ is trivial and $\Lambda_p$ is not.

We have seen a method to possibly overcome this obstruction,
but it relies on optional extra information (the subgroup $G$ of automorphisms of $X$ and/or the Galois group $\Gal(K/k)$) and again it is not bound to work:
in fact the orbits of an element of $\Lambda_p$ under the action of $G$ and/or $\Gal (K/k)$  might not be big enough to show that the element is not in $\ker \iota_p$ (cf.~Proposition~\ref{p:orbit}).
If the sublattice computed by the algorithm is already the full geometric Picard lattice,
then finding more divisors would not improve (i.e. decrease) the size of $\Lambda_p$,
turning the algorithm into an infinite loop.

As for Subsection~\ref{ss:NonSharp}, 
also in this case might happen that in fact $\Lambda=\Pic \Xbar$ but the algorithm fails to prove it.
So again, after terminating the algorithm, 
we might still want to try to prove that the obtained lattice is the full geometric Picard lattice in some other ad hoc way.

\bibliography{MethodK3Biblio}
\bibliographystyle{plain}

\end{document}